\documentclass[11pt,reqno]{amsart}
\usepackage{amsmath,amssymb,amsthm,amsfonts,enumerate,hyperref}
\theoremstyle{plain}
\newtheorem{lemma}{Lemma}[section]
\newtheorem{theorem}{Theorem}[section]
\newtheorem{proposition}[lemma]{Proposition}

\theoremstyle{definition}

\theoremstyle{remark}

\numberwithin{equation}{section}

\newcommand{\Riem}{\mathrm{Rm}}

\begin{document}

\title{Totally umbilical surfaces in three-manifolds with a parallel null vector field}
\author{Wafaa Batat}
\address{Ecole Nationale Polytechnique d'Oran,
	B.P 1523 El M'naouar, 31000 Oran, Algeria}
\email{batatwafa@yahoo.fr}
\author{Stuart James Hall}
\address{School of Mathematics and Statistics, Herschel Building, Newcastle University, Newcastle upon Tyne, NE1 7RU, UK} 
\email{stuart.hall@ncl.ac.uk}

\maketitle
\begin{abstract}
We study non-degenerate, totally umbilical surfaces of  a special class of pseudo-Riemannian manifolds, namely Walker three-manifolds. We show that such surfaces are either one of a totally geodesic family described by Calvaruso and Van der Veken or the ambient manifold must be locally conformally flat (here the surface can also be totally geodesic). The proof makes use of a key technique deployed by Manzano and Soaum in their recent classification of totally umbilical surfaces in homogeneous Riemannian three-manifolds. 
\end{abstract}
\section{Introduction}
\subsection{Background}

One of the most natural problems in submanifold theory is to find special embeddings of a manifold $M$ into an ambient Riemannian or pseudo-Riemannian manifold $(N,g)$. The conditions that make an embedding special often involve the second fundamental form, a symmetric $(0,2)$-tensor field on $M$. For example, if the second fundamental form vanishes then the embedding is said to be totally geodesic (so geodesics on the submanifold remain as geodesics in the ambient space).  There are a number of ways to generalise this condition; one way would be to look for submanifolds with parallel second fundamental form, yielding the notion of a parallel submanifold. Another, the generalisation we take up in this article, is to require that the traceless component of the second fundamental form vanishes (i.e. the second fundamental form is a multiple of the induced metric on the submanifold $M$); these submanifolds are known as \textit{totally umbilical}.\\
\\
The system of PDEs that correspond to a totally umbilical embedding are difficult to solve in general and so classifiying such submanifolds, even within a fixed ambient manifold $(N,g)$, is often impossible. However, in the case when the ambient manifold has a very explicit and well-understood geometry, there has been more progress. In three dimensions, when the ambient space is a homogeneous Riemannian manifold, there is now a complete classification of totally umbilical surfaces due to recent work of Souam and Toubiana \cite{ST}, and Manzano and Souam \cite{MS}.\\
\\
In the pseudo-Riemannian setting such classification problems appear to be harder, even in the low-dimensional setting. This is in part because of the greater variety of behaviour that pseudo-Riemannian metrics may exhibit. We refer the reader to the article \cite{P} for an introduction to the theory of totally umbilical submanifolds of Lorentzian manifolds. As in the Riemannian case, it is sensible to proceed by studying the submanifold geometry of ambient manifolds of a fixed geometric type.\\
\\
A particularly interesting class of pseudo-Riemannian manifold are ones which admit a parallel null vector field. The study of such metrics in the three-dimensional Lorentzian setting was initiated in \cite{CGRVA}. We will refer to these manifolds as \textit{Walker three-manifolds}.  A useful feature of such metrics is the existence of a local coordinate system $(t,x,y)$ in which the parallel null vector field is $\frac{\partial}{\partial t}$ and the metric takes the form 
\begin{equation}\label{MET}
g_{f}^{\varepsilon} = \varepsilon dx^{2}+f(x,y)dy^{2}+2dtdy,
\end{equation}
where $\varepsilon = \pm 1$ and $f(x,y)$ is a smooth function. The difficulty in providing a classification of special submanifolds for these metrics is that the function $f$ can be essentially arbitrary and so we have to deal with an infinite and highly non-trivial set of ambient geometries. In \cite{CV}, Calvaruso and Van der Veken classified parallel surfaces in these manifolds. A totally umbilical surface where the second fundamental form is a constant multiple of the induced metric would also be parallel. 
\subsection{Results}
In this paper we marry the techniques of Mazano--Toubiana--Souam with the those of Calvaruso--Van der Veken and prove the following result which demonstrates that totally umbilical surfaces can only exist in a  restricted set of ambient Walker manifolds. This is somewhat analogous to the results in the Riemannian case where only certain homogenous three-manifolds admit any totally umbilical surfaces.
\begin{theorem} \label{T1}
Let $(N,g_{f}^{\varepsilon})$ be a Walker three-manifold with metric $g_{f}^{\varepsilon}$ given by (\ref{MET}). Assume that there is no open subset of $N$ where $f_{xx}=0$. Then, if $\Sigma\subset N$ is a totally umbilical surface, one of the following is true:
\begin{itemize}
\item $\Sigma$ is one of the parallel surfaces described in Theorem 2 of \cite{CV}.
\item The ambient manifold $(N,g_{f}^{\varepsilon})$ is locally conformally flat.
\end{itemize}
\end{theorem}
In the case where $(N,g_{f}^{\varepsilon})$ is locally conformally flat, we do not yet know of any non-parallel example of a totally umbilical surface. Constructing and further characterising such surfaces are natural steps that we hope to take up in future work.  The techniques of Manzano, Souam and Toubiana were originally used to study homogenous Riemannian metrics in three-dimensions.  It is also possible to use these ideas to classify certain homogeneous pseudo-Riemannian geometries (as mentioned previously there are many more geometries in the pseudo-Riemannian case).  This work will appear elsewhere \cite{BH17}. \\
\\
\textit{Acknowledgements:} The work for this paper was conducted in December 2016 when WB was visiting SH at the University of Newcastle. This research visit was funded by both the University of Newcastle and the Ecole Nationale Polytechnique d'Oran. We would like to thank both these institutions for their support. We would also like to thank Thomas Murphy for his 
interest, friendship and continued support. 

\section{The geometry of Lorentzian Walker three-manifolds}
We are interested in three-dimesional, pseudo-Riemannian manifolds $(N,g_{f}^{\varepsilon})$ where the metric is given by Equation (\ref{MET}).
Here $N$ can be taken to be an open subset of $\mathbb{R}^{3}$ and $\varepsilon=\pm1$. Following Calvaruso--Van der Veken, it is convenient to encode the metric via an orthonormal basis
$$e_{1} = \partial_{x}, \qquad e_{2} = \frac{2-f}{2\sqrt{2}}\partial_{t}+\frac{1}{\sqrt{2}}\partial_{y}, \qquad e_{3} = \frac{2+f}{2\sqrt{2}}\partial_{t}-\frac{1}{\sqrt{2}}\partial_{y}.$$
In this case we have
$$\langle e_{1},e_{1}\rangle = \varepsilon, \qquad \langle e_{2},e_{2}\rangle =1 \quad \textrm{ and } \quad \langle e_{3},e_{3}\rangle=-1. $$
We shall use $\langle \cdot, \cdot\rangle$ to denote both the ambient Walker metric $g_{f}^{\varepsilon}$ as well as the induced metric on the surface $\Sigma \subset N$. Let $\overline{\nabla}$ and $\nabla$ be the Levi-Civita connections of $g_{f}^{\varepsilon}$ and  the induced metric on $\Sigma$ respectively. The non-zero components of the Levi-Civita connection of $g_{f}^{\varepsilon}$ are given by
\begin{equation}\label{LCcon1}
\overline{\nabla}_{e_{2}}e_{1}=-\overline{\nabla}_{e_{3}}e_{1}=\frac{1}{4}f_{x}(e_{2}+e_{3}), 
\end{equation} 
and
\begin{equation}\label{LCcon2}
\overline{\nabla}_{e_{2}}e_{2}=-\overline{\nabla}_{e_{3}}e_{2}=-\overline{\nabla}_{e_{2}}e_{3}=\overline{\nabla}_{e_{3}}e_{3}=-\frac{\varepsilon}{4}f_{x}e_{1}.
\end{equation}
The convention we use for the curvature is
$$\overline{\Riem}(X,Y)Z = [\overline{\nabla}_{X},\overline{\nabla}_{Y}]Z-\overline{\nabla}_{[X,Y]}Z.$$
The curvature tensor of $g_{f}^{\varepsilon}$ given by
\begin{equation} \label{Rm1}
\overline{\Riem}(e_{1},e_{2})e_{1}=-\overline{\Riem}(e_{1},e_{3})e_{1} = \frac{1}{4}f_{xx}(e_{2}+e_{3}),
\end{equation}
\begin{equation}\label{Rm2}
\overline{\Riem}(e_{1},e_{2})e_{2}=-\overline{\Riem}(e_{1},e_{2})e_{3} = -\overline{\Riem}(e_{1},e_{3})e_{2}=\overline{\Riem}(e_{1},e_{3})e_{3} = -\frac{\varepsilon}{4}f_{xx}e_{1}.
\end{equation}
We will consider non-dengenerate surfaces $\Sigma\subset N$ (i.e. the induced metric is Riemannian or Lorentzian) with a $\delta$-normal vector $\mathcal{V}$ where
$$\mathcal{V} = v_{1}e_{1}+v_{2}e_{2}+v_{3}e_{3} \textrm{ with } \varepsilon v_{1}^{2}+v_{2}^{2}-v_{3}^{2}=\delta=\pm 1,$$
for some $v_{i} \in C^{\infty}(\Sigma)$.  \\
\\
Given two tangent vectors to $\Sigma$, $X$ and $Y$ then
$$\overline{\nabla}_{X}Y = \nabla_{X}{Y} +h(X,Y)\mathcal{V},$$
where $h(X,Y)$ is the \textit{second fundamental form}. The surface $\Sigma$ is said to be \textit{totally umbilical} if 
\begin{equation}\label{TUeqh}
h(\cdot,\cdot)=\delta\lambda \langle \cdot,\cdot\rangle,
\end{equation}
for a function $\lambda \in C^{\infty}(\Sigma)$. This condition can also be phrased in terms of the shape operator $S$. The surface is totally umbilical if
\begin{equation}\label{TUeq}
S(T):=-\overline{\nabla}_{T}\mathcal{V}=\lambda T,
\end{equation}
for any $T \in T\Sigma$. We note that Equations (\ref{TUeqh}) and (\ref{TUeq}) are related by
$$\langle S(X),Y\rangle = \delta h(X,Y),$$
for all $X,Y, \in T\Sigma$.  
\section{The proof of Theorem A}
The proof of Theorem \ref{T1} uses one of the key ideas from both \cite{MS} and \cite{ST}.  This is to compute the derivative of the function $\lambda$ in the direction of the Lie bracket $[e_{1}^{T},e_{2}^{T}]$ where $e_{i}^{T}$ is the projection of the standard frame to the tangent bundle $T\Sigma$.  There are two different ways of doing this, one is to use the Levi-Civita connection of the metric and compute, 
$$\langle \nabla_{e^{T}_{1}}e^{T}_{2}-\nabla_{e^{T}_{2}}e^{T}_{1},\nabla \lambda\rangle.$$
The second way is to compute directly, 
 $$e^{T}_{1}(e^{T}_{2}(\lambda))-e^{T}_{2}(e^{T}_{1}(\lambda)).$$
Of course the result of these two calculations should be identical and requiring this gives something akin to integrability conditions for the PDEs involving $\lambda$ and the angle functions $v_{i}$.\\
\\
As in \cite{CV}, we consider two cases: where the function $v_{1}$ in the vector $\mathcal{V}$ has open neighbourhoods in which it does not and does vanish (of course the former condition always determines an open subset of $\Sigma$). In the case that $v_{1}=0$, this means that the frame $e_{1}$ is always tangent to the surface $\Sigma$. However, the Equations (\ref{LCcon1}) and (\ref{LCcon2}) for the Levi-Civita connection show that
$\overline{\nabla}_{e_{1}}\mathcal{V}$ does not have any component in $e_{1}$. Thus $\lambda=0$ and the surface would be totally geodesic and therefore parallel. Hence we can refer to the results of Calvaruso--Van der Veken in this case who show that no such surfaces can exist. We will henceforth assume that $v_{1}\neq 0$.

\subsection{Computing $\nabla \lambda$}
In what follows it will be useful to compute the derivative of the function $\lambda$ coming from the totally umbilical condition (\ref{TUeq}) in terms the function $f$ appearing in the metric (\ref{MET}) and the functions $v_{i}$ appearing in the description of the $\delta$-normal vector $\mathcal{V}$.\\
\\
Following Calvaruso--Van der Veken, we define two vectors tangent to $\Sigma$
$$T_{1} = v_{1}e_{2}-\varepsilon v_{2}e_{1} \textrm{ and } T_{2} = v_{1}e_{3}+\varepsilon v_{3}e_{1}.$$ 
In other words,the matrix 
$$M = \left( \begin{array}{ccc} -\varepsilon v_{2} & \varepsilon v_{3} & v_{1} \\ v_{1} & 0 & v_{2} \\ 0 & v_{1} & v_{3} \end{array}\right), $$
yields the basis  $\{T_{1},T_{2},\mathcal{V}\}$ in terms of the vectors $e_{i}$.  
The following lemma relates the derivative of the function $\lambda$ to the curvature of the metric $g_{f}^{\varepsilon}$.
\begin{lemma}
Let $\Sigma \subset N$ be a totally umbilical surface and let $\lambda, T_{1},T_{2},$ and $\mathcal{V}$ be as previously, then	
$$\overline{\Riem}(T_{1},T_{2})\mathcal{V} = \langle \nabla \lambda, T_{2}\rangle T_{1}-\langle \nabla \lambda, T_{1}\rangle T_{2}.$$
\end{lemma}
\begin{proof}
	The result follows  by straightforward calculation after noting 
	$$\overline{\Riem}(T_{1},T_{2})\mathcal{V} = \overline{\nabla}_{T_{1}}\overline{\nabla}_{T_{2}}\mathcal{V}-\overline{\nabla}_{T_{2}}\overline{\nabla}_{T_{1}}\mathcal{V}-\overline{\nabla}_{[T_{1},T_{2}]}\mathcal{V} = \overline{\nabla}_{T_{2}}(\lambda T_{1}) - \overline{\nabla}_{T_{1}}(\lambda T_{2}) +\lambda[T_{1},T_{2}],$$
	by Equation (\ref{TUeq}).  
\end{proof}
One can also compute $\overline{\Riem}(T_{1},T_{2})\mathcal{V}$ in terms of the frame $\{e_{i}\}$ by using Equations (\ref{Rm1}) and (\ref{Rm2}). This yields the following:
\begin{lemma} \label{GradLambda}
The following equations hold:
\begin{align}
\langle \nabla \lambda, e_{1} \rangle  = & \ \delta\frac{f_{xx}v_{1}(v_{2}-v_{3})^{2}}{4}, \label{dlamE1} \\
\langle \nabla \lambda, e_{2} \rangle  = & \ \delta\frac{f_{xx}(v_{3}-v_{2})}{4} ( v_{1}^2 + \varepsilon v_{3}(v_{2} - v_{3})), \label{dlamE2}\\
\langle \nabla \lambda, e_{3} \rangle  = & \ \delta\frac{f_{xx}(v_{2}-v_{3})}{4} \left( v_{1}^2 + \varepsilon v_{2}(v_{2} - v_{3})\right) \label{dlamE3}.
\end{align}
\end{lemma}
This lemma allows us to recover some known results. In the case that $f_{xx}=0$ on some open neighbourhood $\Omega\subset N$, the curvature Equations (\ref{Rm1}) and (\ref{Rm2}) vanish and so we have a flat Walker metric (we will assume from here that the manifold $N$ is connected and so cannot have flat and non-flat regions). Lemma \ref{GradLambda} shows that a totally umbilical surface embedded in a flat Walker three-manifold must have constant $\lambda$ and so be a parallel surface. Such surfaces are classified (see for example the paper of Chen and Van der Veken \cite{ChV}).  If the ambient metric is not flat but $\lambda$ is constant and so the surface is parallel, we recover  part of the key lemma (Lemma 3 in \cite{CV}) of Calvaruso--Van der Veken that a parallel surface of $N$ must have normal vector with $v_{1}=0$ or $v_{2}=v_{3}$. 
\subsection{First calculation of the Lie bracket $[e_{1}^{T},e_{2}^{T}](\lambda)$ }
Given a vector field $\eta\in \Gamma (TN)$ we can compute the tangential projection of its restriction to $\Sigma$, which we will denote $\eta^{T}$, as
$$\eta^{T} = \eta-\delta\langle \eta,\mathcal{V}\rangle\mathcal{V}.$$
Hence for the frame field $\{e_{i}\}$ we obtain
$$\langle e_{i}^{T},e_{j}^{T}\rangle = \varepsilon_{i}\delta_{ij}-\delta\varepsilon_{i}\varepsilon_{j}v_{i}v_{j},$$
where $\varepsilon_{1}=\varepsilon$, $\varepsilon_{2}=1 $, and $\varepsilon_{3}=-1$.
We compute
$$[e_{1}^{T},e_{2}^{T}](\lambda) = \langle \nabla_{e_{1}^{T}}e_{2}^{T}-\nabla_{e_{2}^{T}}e_{1}^{T},\nabla \lambda\rangle.$$

\begin{lemma}
	The derivative $[e_{1}^{T},e_{2}^{T}](\lambda)$ is given by
\begin{equation}\label{firliebrac}
[e_{1}^{T},e_{2}^{T}](\lambda) =\frac{\delta(v_{2}-v_{3})f_{xx}}{4}\left(\frac{\varepsilon f_{x}(v_{2}-v_{3})^{2}}{4}(\delta(\varepsilon v_{1}^{2}+v_{2}^{2}-v_{2}v_{3})-1)+\lambda v_{1}\right)
\end{equation}
\end{lemma}
\begin{proof}
	The covariant derivative $\nabla_{e_{i}^{T}}e_{j}^{T}$ can be computed via the formula
\begin{equation*}
\nabla_{e_{i}^{T}}e_{j}^{T} = \nabla_{e_{i}^{T}}(e_{j}-\delta\langle \mathcal{V},e_{j}\rangle \mathcal{V})=(\overline{\nabla}_{e_{i}^{T}}e_{j}-\delta\overline{\nabla}_{e_{i}^{T}}\langle \mathcal{V},e_{j}\rangle\mathcal{V}-\delta\langle \mathcal{V},e_{j}\rangle\overline{\nabla}_{e_{i}^{T}}\mathcal{V})^{T}.
\end{equation*}
Using Equation (\ref{TUeq}) and the fact that $\mathcal{V}^{T}=0$ this can be rewritten as
$$\nabla_{e_{i}^{T}}e_{j}^{T} =  \left(\sum_{k=1}^{k=3}\varepsilon_{k}\langle e_{i}^{T},e_{k}^{T}\rangle (\overline{\nabla}_{e_{k}}e_{j})^{T}\right)+\delta \langle\mathcal{V},e_{j}\rangle\lambda e_{i}^{T}.$$
Straightforward calculation using Equations (\ref{LCcon1}) and (\ref{LCcon2}) yields
\begin{equation*}
\nabla_{e_{1}^{T}}e_{2}^{T} =  \delta\left(\lambda v_{2} + \frac{v_{1}(v_{2}-v_{3})f_{x}}{4} \right)e_{1}^{T},
\end{equation*}
and
\begin{equation*}
\nabla_{e_{2}^{T}}e_{1}^{T} =\left(  \frac{(\delta v_{2}(v_{3}-v_{2}) +1)f_{x}}{4} + \delta\varepsilon\lambda v_{1} \right)e_{2}^{T} +\left( \frac{(\delta v_{2}(v_{3}-v_{2})+1)f_{x}}{4}\right)e_{3}^{T}.
\end{equation*}
The result now follows from using Equations (\ref{dlamE1})-(\ref{dlamE3}).
\end{proof}

\subsection{Second calculation of the Lie bracket}
We note the following Lemma for calculating the derivative of the functions $v_{i}$ appearing as the coefficients in the normal vector $\mathcal{V}$.
\begin{lemma}
\begin{equation} \label{InviEj}
\langle \nabla v_{i},e_{j}^{T}\rangle =\varepsilon_{i} \left(\sum_{k=1}^{k=3}\varepsilon_{k}\langle e_{j}^{T},e_{k}^{T}\rangle\langle \overline{\nabla}_{e_{k}}e_{i},\mathcal{V}\rangle\right)-\varepsilon_{i}\lambda\langle e_{i}^{T},e_{j}^{T}\rangle.
\end{equation}
\end{lemma}
\begin{proof}
By definition,
$$ \langle \nabla v_{i},e_{j}^{T}\rangle =e_{j}^{T}(\varepsilon_{i}\langle\mathcal{V},e_{i}\rangle) = \varepsilon_{i} \langle \overline{\nabla}_{e_{j}^{T}}\mathcal{V},e_{i}\rangle+  \varepsilon_{i} \langle \mathcal{V},\overline{\nabla}_{e_{j}^{T}} e_{i}\rangle. $$
The result follows from Equation (\ref{TUeq}) and expanding $e_{j}^{T}$ in the $\{e_{k}\}$ frame.
\end{proof}
We make a second calculation of the quantity $[e_{1}^{T},e_{2}^{T}](\lambda)$
\begin{lemma}
\begin{equation}\label{secliebrac}
\begin{split}
[e_{1}^{T},e_{2}^{T}](\lambda) = \\
& \frac{\delta(v_{2}-v_{3})f_{xx}}{4}\left(\frac{\varepsilon f_{x}(v_{2}-v_{3})^{2}}{4}(\delta(\varepsilon v_{1}^{2}+v_{2}^{2}-v_{2}v_{3})-1)+\lambda v_{1}\right)-\\
& \frac{\delta(v_{2}-v_{3})}{4}\left(\varepsilon v_{3}(v_{2}-v_{3})f_{xxx}\right).
\end{split}
\end{equation}
\end{lemma}
\begin{proof}
We compute (noting the slight abuse of notation when using the chain rule)
$$e_{i}^{T}(e_{j}^{T}(\lambda)) = e_{i}^{T}(\langle \nabla \lambda, e^{T}_{j} \rangle)=\sum_{k=1}^{3}\frac{\partial \langle \nabla \lambda, e_{j} \rangle}{\partial v_{k}}\langle \nabla v_{k},e_{i}^{T} \rangle + \frac{\partial \langle \nabla \lambda, e_{j} \rangle}{\partial f_{xx}}\left(e_{i}^{T}(f_{xx})\right).$$ 
Using Equations (\ref{dlamE1})-(\ref{dlamE3}) and (\ref{InviEj}) we can thus compute
$$  [ e_{1}^{T},e_{2}^{T}](\lambda) = e_{1}^{T}(e_{2}^{T}(\lambda))-e_{2}^{T}(e_{1}^{T}(\lambda)).$$
A lengthy simplification yields the result.
\end{proof}
Subtracting Equation (\ref{firliebrac}) from Equation (\ref{secliebrac})  gives the main technical tool for proving Theorem \ref{T1}.
\begin{proposition}
Let $\Sigma\subset (N,g_{f}^{\varepsilon})$ be a totally umbilical surface with $\delta$-normal vector 
$$\mathcal{V} = v_{1}e_{1}+v_{2}e_{2}+v_{3}e_{3}.$$
Then 
$$
v_{3}(v_{2}-v_{3})^{2}f_{xxx}=0.
$$
Hence at each point in $\Sigma$, there is a small neighbourhood in which one or more of
$$f_{xxx}=0, \qquad v_{3}=0, \qquad v_{2}=v_{3},$$
must occur.
\end{proposition}

\subsection{Deducing Theorem \ref{T1}}
Suppose $v_{3}=0$ (and of course $v_{1}\neq 0$). Then the vectors 
$$T_{1} = v_{1}e_{2}-\varepsilon v_{2}e_{1} \quad  \textrm{ and } \quad T_{2} = v_{1}e_{3}$$
are a non-degenerate tangent frame to the surface $\Sigma$ and the vector 
$$\mathcal{V}=v_{1}e_{1}+v_{2}e_{2},$$
is a $\delta$-normal vector. As $\Sigma$ is totally umbilical, we have
$$-\overline{\nabla}_{T_{1}}\mathcal{V} = \lambda T_{1}.$$
However
$$\langle\overline{\nabla}_{T_{1}}\mathcal{V},e_{3}\rangle =- v_{1}^{2}\frac{f_{x}}{4} .$$
Hence $f_{x}=0$ which would yield a flat ambient manifold.\\
\\
Suppose $v_{2}=v_{3}$ then  $\lambda$ is constant by Equations (\ref{dlamE1})-(\ref{dlamE3}).  This would mean $M$ is a parallel surface with $v_{2}=v_{3}$. These surfaces are classified by Calvaruso--Van der Veken \cite{CV}. In this case the normal vector $\mathcal{V}$ can be written as
$$\mathcal{V} = \eta\partial_{x} +\sqrt{2}v_{2}\partial_{t},$$
where $\eta=\pm1$ depending upon the values of $\varepsilon$ and $\delta$.
The surfaces can be described explicitly in local coordinates,
$$\iota(u,v)=(u,x(v),v),$$ 
where $x(v)$ satisfies the equations
$$ x'(v)=-\eta\varepsilon\sqrt{2}v_{2}(v),$$
and
$$x''(v)-\frac{\varepsilon}{2}f_{x}(x(v),v)=C,$$
for some constant $C$.\\
\\
If $f_{xxx}=0$ then clearly 
$$f(x,y) = A(y)x^{2}+B(y)x+C(y).$$
In \cite{CGRVA} it is demonstrated that this condition completely characterises locally conformally flat Walker three-manifolds. Hence Theorem A is demonstrated.

\bibliography{TUwalk}
\bibliographystyle{acm}

\end{document}